\documentclass[11pt,reqno]{amsart}
\usepackage{amsthm,amssymb,amsmath}
\usepackage{textcomp}
\usepackage[foot]{amsaddr}
\usepackage{xcolor}
\usepackage{hyperref}
\usepackage[T1]{fontenc}
\usepackage{lmodern}
\usepackage{lmodern}
\hypersetup{
	colorlinks =true , % Non coloriage des liens hypertextes
	linkcolor =blue,
	urlcolor = magenta,
	citecolor =blue
}

\newtheorem{theorem}{Theorem}[section]
\newtheorem*{theorem*}{Theorem}
\newtheorem{lemma}{Lemma}[section]
\newtheorem{corollary}{Corollary}[section]

\theoremstyle{definition}

\newtheorem{definition}{\bf Definition}[section]
\newtheorem*{definition*}{\bf Definition}

\newtheorem{remark}{Remark}
\newtheorem*{remark*}{Remark}

\newtheorem*{example*}{\bf Example}

\title[ Linking theorem with application to a system of coupled Poisson equations ]{A new infinite-dimensional Linking theorem with application to a system of coupled Poisson equations}

\begin{document}

    \author{Ablanvi Songo}

\address{Universit\'{e} de Sherbrooke, D\'{e}partement de math\'{e}matiques, Sherbrooke, Qu\'{e}bec, Canada}

\email{ablanvi.songo@usherbrooke.ca}

\author{ Fabrice Colin }

\address{Laurentian University, School of Engineering and Computer Science, Sudbury, Ontario, Canada}

\email{fcolin@laurentian.ca}



    \begin{abstract}
	Using the minimax technique from the critical point theory, which consists in constructing or transforming a suitable class of applications such that a critical value $c$ of a functional $f$ can be characterized as a minimax value over this class, we establish a new natural infinite-dimensional linking theorem for strongly indefinite functionals by using the $\tau-$topology of Kryszewski and Szulkin. Our result is a generalization of the classical linking theorem \cite[Theorem 2.21]{Wi}. As an application, we obtain the existence of a nontrivial solution to a system of coupled Poisson equations.	
\end{abstract}
\maketitle
\section{Introduction}
In 1978, Paul H. Rabinowitz established the following linking theorem (the reader can find this version in the book of Willem \cite[Theorem 2.21]{Wi}). Let $X=Y\oplus Z$ be a Banach space with $\dim Y <\infty$. Let $\rho>r>0$ and $z_0\in Z$ be such that $\|z_0\|=r$. Define 
\begin{eqnarray*}
				\label{eq 7}
				M &:=& \Big\{ u=y+\lambda_0 z_0\;|\; \| u\| \le \rho, \;\lambda_0 \ge 0,\; y \in Y\Big\}, \\
				\label{eq 8}
				\partial M &:=& \Big\{ u= y + \lambda_0 z_0 \;|\; y \in Y,\; (\|u\| =\rho \; \text{and}\; \lambda_0 \ge 0 )\; \text{or}\; (\|u\| \le \rho \; \text{and}\;  \lambda_0 =0 ) \Big\},\\
				\label{eq 9}
				N &:=& \Big\{ u \in Z \;|\; \|u\| =r \Big\}.
			\end{eqnarray*}
			Let $J \in \mathcal{C}^1 (X, \mathbb{R})$ be such that
			\begin{center}
				$b:= \underset{u\in N}{\inf} \; J(u) > a := \underset{u \in \partial M}{\max}\; J(u) $. 
			\end{center}
			If $J$ satisfies the $(PS)_c$ condition (see Definition $\ref{def 2.3}$) with
			\begin{equation*}
				c :=\underset{\gamma \in \Gamma}{\inf} \; \underset{u\in M}{\max}\; J(\gamma(u)),
                \end{equation*}
                \begin{equation*}
                    \Gamma := \Big\{ \gamma\in \mathcal{C}(M, X)\;|\; \gamma\Big|_{\partial M} =\operatorname{id}\Big\},
                \end{equation*}
then $c$ is a critical value of $J$.\\

In order to apply the above theorem to find a solution to variational problems (based on PDE), it is crucial that the dimension of the space $Y$ must be finite in the decomposition of the space $X$ on which the associated functional (with the problem) is defined.\\

Suppose that, one is interested in finding a nontrivial solution to the following systems of two semilinear coupled Poisson equations 
		\begin{equation}
			(\mathcal{P})
			\begin{cases}
				-\Delta u = \lambda u + g(x,v),\; x \in \Omega,\\
				-\Delta v = \delta v + f(x,u), \;x \in \Omega, \\
				u,v \in H^{1}_{0}(\Omega), 
			\end{cases}
		\end{equation}
		where $\Omega$ is an open bounded subset of $\mathbb{R}^N$, $  f(x,t)$ and $ g(x,t)$ are nonlinear continuous functions in $\bar{\Omega}\times \mathbb{R}$, $\lambda$ and $\delta$ are real numbers and $\Delta$ is the Laplace operator. The above theorem cannot be applied. Indeed, since the nonlinear functions $f(x,t)$ and $g(x,t)$ have a growth in $t$, the problem $(\mathcal{P})$ has a variational structure. In fact, the problem $(\mathcal{P})$ constitutes the Euler-Lagrange equation for the functional
		\begin{equation}
			J(u,v) := \int_{\Omega} \Bigg(\nabla u .\nabla v - \dfrac{\lambda u^2}{2}- \dfrac{\delta v^2}{2}-F(x,u)-G(x,v) \Bigg)\; dx,
		\end{equation}
		where 
		\begin{equation*}
			F(x,u) := \int_{0}^{u(x)} f(x,t)dt, \qquad G(x,v) := \int_{0}^{v(x)} g(x,t)dt.
		\end{equation*}
		Moreover, the first part of the functional $J$ is a strongly indefinite operator. In particular, $J$ has the form 
		\begin{equation*}
			\dfrac{\|Q(u,v)\|^2}{2} -\dfrac{\|P(u,v)\|^2}{2} -\Phi(u,v),
		\end{equation*}
		where $(u,v)$ belongs to a Hilbert space $X= Y\oplus Z$ and where $Y$ and  $Z$ are both infinite-dimensional subspaces and where $P$ and $Q$ are respectively the orthogonal projections on $Y$ and $Z$. The problem $(\mathcal{P})$ is then strongly indefinite, because the energy functional associated with problem $(\mathcal{P})$ has a strong indefinite quadratic part. There is no more mountain pass structure but a linking one. Therefore, the proofs of the main results related to the functional $J$ cannot rely on classical min-max results. Namely, the classical linking theorem could not be used to find a nontrivial solution to the problem $(\mathcal{P})$. \\
        
         We would like to point out that there are various infinite-dimensional generalizations of the classical linking theorem in the literature that can help find a nontrivial solution to the problem $(\mathcal{P})$. Among others, we can mention \cite{KS, BB, CW}.\\
         
        In this paper, we propose a new way to generalize the classical linking theorem. To do so, we consider on the Hilbert space $X$ the $\tau-$topology introduced by Kryszewski and Szulkin \cite{KS}. We replace the set $\Gamma$ in the above theorem by the following set\\
         \begin{center}
		$\Gamma : = \Big\{\gamma : M \rightarrow X \;\Big|\;\gamma \quad \text{is}\quad \tau-$continuous, $\gamma \big|_{ \partial M} =\operatorname{id}$ and\\ every  $u\in \operatorname{int}(M)$ has a $\tau-$neighborhood $N_u$ in $X$ such that $(\operatorname{id}-\gamma)(N_u\cap \operatorname{int}(M))$ \\ is contained in a finite-dimensional subspace of $X \Big \}$.
			\end{center}
            
In addition, we suppose that $J$ is $C^1-$functional such that $J$ is $\tau-$upper semi-continuous, $\nabla J$ is weakly sequentially continuous, $J$ satisfies the $(PS)_c$ condition and $J$ fulfills the linking geometry
\begin{equation*}
				\underset{u\in N}{\inf} \; J(u) >\underset{u \in \partial M}{\sup}\; J(u) . 
			\end{equation*}
Then, we show that 	\begin{equation*}
				c :=\underset{\gamma \in \Gamma}{\inf} \; \underset{u\in M}{\sup}\; J(\gamma(u)),
                \end{equation*}
                is a critical value of the functional $J$; making our infinite-dimensional result a natural way to generalize the classical linking theorem (see Theorem $\ref{theorem 3.1}$).	\\
The proof of our abstract result is based on an infinite-dimensional general minimax result proved by the two authors in \cite{CS} (see \cite[Theorem 2.1]{CS}).\\
        Our result can be applied to a wider class of indefinite functionals,
 especially to strongly indefinite functionals; that is, functionals of the form
 \begin{equation*}
     J(u) = \dfrac{1}{2} \langle Lu, u \rangle - \Psi(u)
 \end{equation*}
defined on a Hilbert space $X$, where $L : X \to X$ is a self-adjoint operator with negative and positive eigenspaces that are infinite-dimensional.\\
To test the value of our result, we will apply it to find a nontrivial solution to the problem $(\mathcal{P})$.\\

	There are many papers in the literature that are devoted to the study of the existence of solutions to these problems, either in a whole space or in bounded (also unbounded) domains; see for example \cite{FDR, CF, ZL, FY, LY}.\\

The remaining of the paper is organized as follows. In the next Section, we recall some classical results and present some other preliminary results which will be used in the following. In Section 3, we state and prove our main result and in Section 4, we apply our main result to obtain a nontrivial solution to the problem $(\mathcal{P})$.

\section{Kryszewski-Szulkin degree theory}
In this section, we are following the presentation of the degree theory of Kryszewski and Szulkin given in \cite{Wi} by Michel Willem.

Let $Y$ be a real separable Hilbert space endowed with inner product $( \cdot, \cdot )$ and the associated norm $\|\cdot\|$.  

On $Y $ we consider the $\sigma-$topology introduced by Kryszewski and Szulkin; that is, the topology generated by the norm 
\begin{equation}
	|u|_\sigma := \sum_{k =0}^{\infty} \frac{1}{2^ {k+1}}|(u, e_k)|, \quad u \in Y,
\end{equation}
where $(e_k)_{k\ge 0}$ is a total orthonormal sequence in $Y$.\\

\begin{remark}	
	\label{remark 1}
	By the Cauchy-Schwarz inequality, one can show that $|u|_\sigma \le \|u\|$ for every $u \in Y$. Moreover, if $(u_n)$ is a bounded sequence in $Y$ then 
	\begin{equation*}
		u_n \rightharpoonup u \Longleftrightarrow u_n \overset{\sigma}{\rightarrow} u,
	\end{equation*}
	where $\rightharpoonup$ denotes the weak convergence and $ \overset{\sigma}{\rightarrow}$ denotes the convergence in the $\sigma-$topology.
\end{remark}
Let $U$ be an open bounded subset of $Y$ such that its closure $\bar{U}$ is $\sigma-$closed.
\begin{definition}[\cite{Wi}]
	\label{definition 2.1}
	A map $f : \bar{U} \rightarrow Y$ is $\sigma-$admissible (admissible for short) if\\
	(1) $f$ is $\sigma-$continuous, \\
	(2) each point $u \in U$ has a $\sigma-$neighborhood $N_u$ in $Y$ such that $(id-f)(N_u\cap U)$ is contained in a finite-dimensional subspace of $Y$.
\end{definition}
\begin{definition}[\cite{Wi}]
	A map $h :[0,1]\times \bar{U}\rightarrow Y$ is an admissible homotopy if \\
	(1)  $0 \notin h([0,1] \times \partial U)$,\\
	(2) $h$ is $\sigma-$continuous, that is $t_n \rightarrow t$ and $u_n \overset{\sigma}{\rightarrow} u$ implies $h(t_n,u_n) \overset{\sigma}{\rightarrow} h(t,u)$,\\
	(3) $h$ is $\sigma-$locally finite-dimensional. That is, for any $(t,u) \in[0,1]\times U$ there is a neighborhood $N_{(t,u)}$ in the product topology of $[0,1]$ and $(X, \sigma)$ such that $\{ v-h(s,v) \; | \; (s,v) \in N_{(t,u)} \cap ([0,1]\times U)\}$ is contained in a finite-dimensional subspace of $Y$.
\end{definition} Then, for an admissible map $f$ such that $0 \notin f(\partial U)$   we have (see \cite[ Theorem 6.6]{Wi})
\begin{equation*}
	deg (h(0,.), U) = deg(h(1,.),U),
\end{equation*}
where $deg$ is the topological degree of $f$ (about 0). Such a degree possesses the usual properties; in particular, if $f : \bar{U} \rightarrow Y$ is admissible with $0 \notin f(\partial U)$ and $deg(f, U)\ne 0$, then there exists $u \in U$ such that $f(u)= 0$.\\

Now, let $X = Y\oplus Z$, where $Y$ is closed and $Z=Y^\perp$, be a real separable Hilbert space endowed with inner product $\langle \cdot, \cdot \rangle$ and the associated norm $\|\cdot\|$. Let $(e_k)_{k\ge 0}$ be an orthogonal basis of $Y$. On $X$, we define a new norm by setting
\begin{equation*}
	|u|_\tau := \max \Bigg(\sum_{k =0}^{\infty} \frac{1}{2^ {k+1}}|\langle Pu, e_k\rangle|, \|Qu\| \Bigg), \quad u \in X,
\end{equation*}
where $P$ and $Q$ are respectively orthogonal projections of $X$ into $Y$ and $Z$ and we denote by $\tau$  the topology generated by this norm. The topology $\tau$ was introduced by Kryszewski and Szulkin \cite{KS}.
\begin{remark}	
	\label{remark 2} For every $u\in X$, 
	we have $\|Qu\| \le |u|_\tau$ and $|Pu|_\sigma \le |u|_\tau$. Moreover, if $(u_n)$ is a bounded sequence in $X$ then 
	\begin{equation}
    \label{eq 4}
	u_n \overset{\tau}{\rightarrow} u \Longleftrightarrow	Pu_n \rightharpoonup Pu \quad\text{and}\quad Qu_n \to Qu,
	\end{equation}
	where $\rightharpoonup$ denotes the weak convergence and $ \overset{\tau}{\rightarrow}$ denotes the convergence in the $\tau-$topology.
\end{remark}

A functional $J: X \to \mathbb{R}$ is said to be $\tau-$upper semi-continuous if the set $$J_{\beta}:= \Big\{u \in X \;|\; J(u)\ge \beta \Big\}$$ is $\tau-$closed and 
we say that $\nabla J$ is weakly sequentially continuous if the sequence $(\nabla J(u_n))$ converges weakly to $\nabla J(u)$ whenever $(u_n)$ converges weakly to $u$ in $X$.\\

We consider the class of $\mathcal{C}^1-$functionals $J : X \rightarrow \mathbb{R}$ such that 

\text{(A)}	$J$ is $\tau-$upper semi-continuous and $\nabla J$ is weakly sequentially continuous.\\

\begin{definition}
\label{def 2.3}
A $\mathcal{C}^1-$functional $J$ is said to satisfy the $(PS)_c$ condition (or the Palais-Smale condition at level $c$) if any sequence $(u_n)\subset X$ such that 
\begin{equation*}
	J(u_n) \rightarrow c \quad \textit{and} \quad J' (u_n) \rightarrow 0,
\end{equation*}
has a convergent subsequence.
\end{definition}

We recall the following two results from \cite{CS} that will play a key role in the proof of our abstract result; see \cite[Theorem 2.1 ]{CS} and \cite[Corollary 3.1]{CS}.

\begin{theorem}[General minimax principle]
	\label{theorem 2.1}
Assume that $J$ satisfies $(A)$, that is $J$ is $\tau-$upper semi-continuous and $\nabla J$ is weakly sequentially continuous. Let $M$ be a closed metric subset of $X$, and let $M_0$ be a closed subset of $M$. 
Let us define 
 \begin{center}
 	$\Lambda_0 : = \Big\{\gamma_0 : M_0 \rightarrow X \;\Big|\;\gamma_0 $\quad \text{is}\quad $\tau-$continuous \Big \},
 \end{center}
\begin{center}
		$\Gamma : = \Big\{\gamma : M \rightarrow X \;\Big|\;\gamma \quad \text{is}\quad \tau-$continuous, $\gamma \big|_{ M_0} \in \Lambda_0$ and\\ every  $u\in \operatorname{int}(M)$ has a $\tau-$neighborhood $N_u$ in $X$ such that $(\operatorname{id}-\gamma)(N_u\cap \operatorname{int}(M))$ \\ is contained in a finite-dimensional subspace of $X \Big \}$.
	\end{center}
	If $J$ satisfies
	\begin{equation}
		\label{eq 7}
		\infty > c := \underset{\gamma \in \Gamma}{\inf}\; \underset{u \in M}{\sup} \;J(\gamma(u))> a := \underset{\gamma_0 \in \Lambda_0}{\sup}\; \underset{u \in M_0}{\sup} \; J(\gamma_0(u)),
	\end{equation}
	Then,
	for every $\varepsilon \in ] 0, \frac{c-a}{2}[$, $\delta >0$ and $\gamma \in \Gamma$ such that
	\begin{equation}
		\label{eq 8}
		\underset{u\in M}{\sup}\; J \circ \gamma (u) \le c +\varepsilon , 
	\end{equation}
	there exists $u \in X$ such that 
    \begin{enumerate}
        \item[$(a)$] \[c-2\varepsilon \le J(u) \le c+2\varepsilon ,\]
        \item[$(b)$] \[dist(u, \gamma(M)) \le 2 \delta,\]
        \item[$(c)$] \[\|J' (u)\| < \frac{8\varepsilon}{\delta}.\]
    \end{enumerate}
\end{theorem}
\begin{corollary}
	\label{cor 1.1}
	Suppose that the assumptions of Theorem $\ref{theorem 2.1}$ are satisfied and suppose that $J$ satisfies $(\ref{eq 7})$. Then there exists a sequence $(u_n)\subset X$ satisfying 
	\begin{equation*}
		J(u_n) \rightarrow c, \qquad J'(u_n) \rightarrow 0.
	\end{equation*}
	In particular, if $J$ satisfies $(PS)_c$ condition, then $c$ is a critical value $J$.
\end{corollary}

	\section{A new generalized linking theorem}
In this section we state and proof the main result of this paper, which is an infinite-dimensional generalization of the classical Linking Theorem \cite[Theorem 2.12]{Wi}.
		\begin{theorem}[Generalized Linking Theorem]
			\label{theorem 3.1}
			Let $X =Y \oplus Z$ be a Hilbert space with $Y$ a closed separable subspace of $X$ which could be infinite-dimensional and $Z := Y^\perp$. Assume that $J$ satisfies $(A)$, that is $J$ is $\tau-$upper semi-continuous and $\nabla J$ is weakly sequentially continuous. Let $\rho > r >0$ and $z \in Z$ be such that $\|z\| =r$. \\ We  define
			\begin{eqnarray}
				\label{eq 7}
				M &:=& \Big\{ u=y+\lambda_0 z\;|\; \| u\| \le \rho, \;\lambda_0 \ge 0,\; y \in Y \Big\}, \\
				\label{eq 8}
				\partial M &:=& \Big\{ u= y + \lambda_0 z \;|\; y \in Y,\; (\|u\| =\rho \; \text{and}\; \lambda_0 \ge 0 )\; \text{or}\; (\|u\| \le \rho \; \text{and}\;  \lambda_0 =0 )\Big \},\\
				\label{eq 9}
				N &:=& \Big\{ u \in Z \;|\; \|u\| =r \Big\}.
			\end{eqnarray}
			Let $J \in \mathcal{C}^1 (X, \mathbb{R})$ such that
			\[
				b:= \underset{u\in N}{\inf} \; J(u) > a := \underset{ u \in \partial M}{\sup}\; J(u). \qquad(\mathcal{G})\]
			If $J$ satisfies the $(PS)_c$ condition with
			\begin{equation*}
				c :=\underset{\gamma \in \Gamma}{\inf} \; \underset{u\in M}{\sup}\; J(\gamma(u)),
                \end{equation*}
                \begin{center}
		$\Gamma : = \Big\{\gamma : M \rightarrow X \;\Big|\;\gamma \quad \text{is}\quad \tau-$continuous, $\gamma \big|_{ \partial M} =\operatorname{id}$ and\\ every  $u\in \operatorname{int}(M)$ has a $\tau-$neighborhood $N_u$ in $X$ such that $(\operatorname{id}-\gamma)(N_u\cap \operatorname{int}(M))$ \\ is contained in a finite-dimensional subspace of $X \Big \}$,
			\end{center}
			then $c$ is a critical value of  $J$.
		\end{theorem}
        	\begin{proof}
		Observe that, $M$ is $\sigma-$closed. 
		In order to apply Corollary $\ref{cor 1.1}$, we only have to verify that $c\ge b$. Let us prove that for every $\gamma \in \Gamma, $ \[\gamma(M)\cap N \ne \emptyset.\] Denote by $P$ the projection of $X$ onto $Y$ such that $PZ = \{0\}$ and let $\gamma \in \Gamma$. Since the topology $\sigma$ is induced by the topology $\tau$ on $Y$, we conclude  that \[P\gamma : M \rightarrow Y\] is  $\sigma-$admissible.\\
			We consider the following map
			\begin{eqnarray*}
				H : [0,1]\times M &\rightarrow& Y\oplus \mathbb{R}.z\\  (t,u) &\mapsto& \Big(tP(\gamma(u))+(1-t)y \Big)+ \Big(\dfrac{t}{r} \| \gamma(u)-P(\gamma(u)) \| +(1-t)\lambda_0 -1\Big)z .
			\end{eqnarray*}
			We claim that the map $H$ is an admissible homotopy. Indeed: 
            \begin{enumerate}
                \item 
			 $H$ is $\sigma-$continuous. Let $t_n\rightarrow t$ in $[0,1]$ and let $u_n\overset{\sigma}{\rightarrow }u$ in $M$. By $(\ref{eq 4})$ and since $|z|_\sigma \le \|z\| =r$, we have
\begin{multline*}
    |H(t_n,u_n)-H(t,u)|_{\sigma} = \Big| \Big( t_nP(\gamma(u_n))- tP(\gamma(u)) \Big) +(t_n-t)(-y) \\+ \frac{1}{r}\Big( t_n \|\gamma(u_n)-P\gamma(u_n)\| -t \|\gamma(u)-P\gamma(u)\|-\lambda_0(t_n-t)\Big) z \Big|_\sigma \\
    \le \Big| t_n P(\gamma(u_n))- tP(\gamma(u))\Big|_\sigma +|t_n-t|\Big|-y-\lambda_0 z\Big|_\sigma \\ +\frac{1}{r} \Big| t_n \|\gamma(u_n)-P\gamma(u_n)\| -t \|\gamma(u)-P\gamma(u)\| \Big||z|_\sigma \\
    \le \Big| t_nP(\gamma(u_n))-tP(\gamma(u))\Big|_\sigma + \rho|t_n-t| \\ + \Big| t_n \|\gamma(u_n)-P\gamma(u_n)\| -t \|\gamma(u)-P\gamma(u)\| \Big| \;
    \underset{ n \to \infty}{ \longrightarrow 0}.
\end{multline*}
			\item  Let $(t,u)\in [0,1]\times \partial M $. Since $\gamma =\operatorname{id}$ on $\partial M$ and that $Pu = y$, then we have \[H(t,u)= y +(\lambda_0 -1)z \ne 0.\] Hence, $0\notin H([0,1]\times \partial M)$.\\
			\item Let $u \in \operatorname{int}(M)$. Since $\gamma \in \Gamma$, then $u$ possesses a $\tau-$neighborhood $N_u$ such that \[(\operatorname{id}-\gamma)(N_u \cap \operatorname{int}(M))\] is contained in a finite-dimensional subspace $F_0$ of $X$.
            
			Let $t\in [0,1]$. We have
			\begin{equation*}
				u-H(t,u)= tP(u-\gamma(u)) + \Big(-t\|\gamma(u)-P(\gamma(u))\|+t\lambda_0 +1\Big)z.
			\end{equation*}
Therefore, every point $(t,u)\in [0,1]\times \operatorname{int}(M)$ has a neighborhood $N_{(t,u)}:= [0,1]\times N_u$ in the product topology of $[0,1]$ and $(X, \sigma)$ such that  \[ \Big\{v-H(s,v) \; |\; (s,v)\in N_{(t,u)}\cap ([0,1]\times \operatorname{int}(M))\Big\}\] is contained in  $F_0 + \mathbb{R}.z$ which is a finite-dimensional subspace of $X$.
\end{enumerate}
			Since the map $H$ is an admissible homotpy, by homotpy invariance property of Kryszewski and Szulkin degree, we have
			\begin{equation*}
				deg\Big(H(1,.), int(M)\Big)=deg\Big(H(0,.), int(M)\Big)
			\end{equation*}
			where 
			\begin{equation*}
			H(1,u) = P(\gamma(u))+ \Big(\frac{1}{r}\|\gamma(u)-P(\gamma(u))\| -1\Big)z,
			\end{equation*}
		and  \begin{equation*}
			H(0,u)= y+(\lambda-1)z = u-z.
		\end{equation*}
			Since $z\in \operatorname{int}(M)$ $\Big( \text{because}\; \|z\|=r \in (0,\rho)\Big)$, by normalization property of Kryszewski and Szulkin degree, 
			\begin{equation*}
			deg\Big(H(1,.), \operatorname{int}(M)\Big)=deg\Big(H(0,.), \operatorname{int}(M)\Big)=deg\Big(u-z, \operatorname{int}(M)\Big) =1 \ne 0. 
		\end{equation*}
		
		Finally, the existence property implies that there is $\bar{u}\in \operatorname{int}(M)$ such that $H(1,\bar{u})=0$. Simultaneously, we obtain \[P(\gamma(\bar{u}))=0\quad \text{and}\quad\|\gamma(\bar{u})\|=r.\] Hence, \[\gamma(M)\cap N \ne \emptyset.\]
		
		 Consequently,
		there exists $u_0 \in \gamma(M)\cap N$ such that 
		\begin{equation*}
			\underset{u \in N}{\inf} J (u) \le J(u_0) \le \underset{u \in M}{\sup}\; J(\gamma(u)).
		\end{equation*}
		That is, 
		\begin{equation*}
			b:=\underset{u \in N}{\inf} J(u) \le  \underset{\gamma \in \Gamma}{\inf} \; \underset{u \in M}{\sup} \; J (\gamma(u))=: c.
		\end{equation*}
		By Corollary $\ref{cor 1.1}$, $c:= \underset{\gamma \in \Gamma}{\inf} \; \underset{u \in M}{\sup} \; J (\gamma(u))$ is a critical value of $J$.
		\end{proof}
        	\section{Application}
		In this section, we apply our main result, namely the generalized linking theorem, that is, Theorem $\ref{theorem 3.1}$ to obtain a nontrivial solution to problem $(\mathcal{P})$. \\
        
        Let $\Omega$ be an open bounded subset of $\mathbb{R}^N$. 
		Here, we assume the following conditions: \\
        
	\begin{enumerate}
		\item[$(H_1)$] The functions $ f$ and $g$ are in $\mathcal{C}(\overline{\Omega}\times \mathbb{R}, \mathbb{R})$ and there exists a constant $c>0$ such that 
        \begin{equation*}
            |f(x,t)|, |g(x,t)|\le c (1+ |t|^{p-1}), \;\; \text{for all}\;\;x\in \overline{\Omega}\;\;\text{and}\;\;t \in \mathbb{R},
        \end{equation*}
    where $2<p<\infty$ if $N=1,2$ and $1<p<2^*:= \dfrac{2N}{N-2}$ if $N\ge 3$.
		\item[$(H_2)$] $\lim\limits_{t \rightarrow 0} \dfrac{f(x,t)}{t}= \lim\limits_{t \rightarrow 0}\dfrac{g(x,t)}{t}=0$, uniformly with respect to $ x \in \Omega$.\\
		
		\item[$(H_3)$] There exists $ 2 <\mu < 2^*$ and $R>0$ such that $|t|\ge R$ we have $0< \mu F(x,t)\le tf(x,t)$ and $0<\mu G(x,t)\le tg(x,t) $.\\
	\end{enumerate}	
		
		In the sequel $N\ge 3$.  We denote by $X:= H^{1}_{0}(\Omega) \times H^{1}_{0}(\Omega)$, the Hilbert space endowed with the inner product and the corresponding norm (see \cite{CF}) :
		\begin{equation}
			\langle (u,v), (u_1,v_1)\rangle := \int_{\Omega} \Big(\nabla u. \nabla u_1 + \nabla v . \nabla v_1 \Big)\;dx, \qquad \|(u,v),(u,v)\|^2 = \| u \|^2 + \|v\|^2.
		\end{equation}
	If we define
		\begin{equation*}
			Y := \Big\{(-v,v) \in X \Big\} \quad \text{and}\quad  Z:=\Big\{(u,u)\in X\Big\},
		\end{equation*}
 since we can write $(u,v)$ as
		\begin{equation*}
			(u,v) = \dfrac{1}{2}(u+v,u+v)+\dfrac{1}{2}(-v+u, v-u),
		\end{equation*}
        then, $X= Y\oplus Z$.\\
Let us denote by $P$ the projection of $X$ onto $Y$ and by $Q$ the projection of $X$ onto $Z$.\\
        
		Finally, let us define the functional $J: X \rightarrow \mathbb{R}$ given by
		\begin{eqnarray*}
			J(u,v) &:=& \int_{\Omega} \Bigg(\nabla u .\nabla v -\dfrac{\lambda u^2}{2} -\dfrac{\delta v^2}{2} \Bigg)\; dx- \int_{\Omega} \Big( F(x,u) + G(x,v)\Big)\; dx \quad (\star)\\
			&=& \dfrac{\|Q(u,v)\|^2}{2} -\dfrac{\|P(u,v)\|^2}{2} - \varphi(u,v),
		\end{eqnarray*}
		where
		\begin{equation*} 
			\varphi(u,v) := \int_{\Omega} \Bigg(\dfrac{\lambda u^2}{2} +\dfrac{\delta v^2}{2}+ F(x,u) + G(x,v)\Bigg) \; dx.\\
		\end{equation*}

Throughout this section, $|\cdot|_p$ represents the usual $L^p$ norm.
		We denote $\rightarrow$  the strong convergence and $\rightharpoonup$ the weak convergence.\\
        
        Here is the main result of this section:
			\begin{theorem}
            \label{thm 4.1}
		Under assumptions $(H_1),(H_2)$ and $(H_3)$, the problem $(\mathcal{P})$ has at least a nontrivial solution.
		\end{theorem}

	Before we give the proof of Theorem $\ref{thm 4.1}$, we shall prove some results related to the functional $J$ on $X$. \\

		\begin{lemma}
			\label{lemme 4.1}
			The functional $J$ given in $(\star)$ is $\mathcal{C}^1$ on $X$. Moreover, for every $(u,v), (w,z) \in X$,
			\begin{equation*}
				\langle J'(u,v), (w,z) \rangle = \int_{\Omega} \Big(\nabla u \nabla z +\nabla v \nabla w -\lambda u w  -\delta v z -f(x,u) w - g(x,v)z \Big)dx.
			\end{equation*}
			Consequently, the weak solutions of problem $(\mathcal{P})$ are exactly the critical points of $J(u,v)$ in $X$.
		\end{lemma}
		
		\begin{proof}
			\textbf{Existence of Gateaux derivative}: Let $(u,v), (w,z) \in X$. For $x\in \Omega$ and $|t|\in (0,1)$ we have 
			\begin{align*}
				\dfrac{1}{t}\Big( \varphi(u+tw,v+tz) -\varphi (u,v) \Big) &= \dfrac{\lambda}{2}\int_{\Omega} \dfrac{1}{t}\Big((u+tw)^2 -u^2\Big)dx + \dfrac{\delta}{2} \int_{\Omega}\dfrac{1}{t}\Big((v+tz)^2 -v^2\Big)dx \\
				&+ \int_{\Omega} \dfrac{1}{t} \Big(F(x,u+tw)-F(x,u)\Big)dx \\
				& +\int_{\Omega} \dfrac{1}{t}\Big(G(x, v+tz)-G(x,v)\Big)dx.
			\end{align*}
			By the mean value theorem, there exists $\alpha \in (0,1)$ such that
			\begin{eqnarray}
				\label{eq 11}
				\dfrac{|(u+tw)^2- u^2|}{|t|} &=& 2 |(u+\alpha t w)w|\nonumber \\
				&\le& 2(|u|+|w|)|w|,
			\end{eqnarray}
			and by the assumption $(H_1)$, we have 
			\begin{eqnarray}
				\label{eq 12}
				\dfrac{|F(x,u+tw) - F(x,u)|}{|t|} &=& |f(x, u+\alpha t w)||w| \nonumber\\
				&\le& c(1+|u+\alpha t w|^{p-1})|w| \nonumber \\
				&\le& c(1+(|u|+ |w|)|^{p-1})|w| \nonumber \\
				&\le& c (1+ 2^{p-1}(|u|^{p-1}+ |w|^{p-1}))|w|.
			\end{eqnarray}
			Similarly, we have
			\begin{eqnarray}
				\label{eq 13}
				\dfrac{|(v+tz)^2- v^2|}{|t|} &\le& 2(|v|+|z|)|z|,
			\end{eqnarray}
			\begin{eqnarray}
				\label{eq 14}
				\dfrac{|G(x,v+tz) - G(x,v)|}{|t|} &\le& c (1+ 2^{p-1}(|v|^{p-1}+ |z|^{p-1}))|z|.
			\end{eqnarray} By the Hölder inequality, the terms on the right hand of inequalities $(\ref{eq 11})$, $(\ref{eq 12})$, $(\ref{eq 13})$ and $(\ref{eq 14})$ are in $L^1(\Omega)$.\\
			On the other hand, 
			\begin{equation*}
				\int_{\Omega} \Big(\nabla (u+tw) \nabla (v +t z)- \nabla u \nabla v\Big) dx = \int_{\Omega} \Big(t\Big(\nabla u \nabla z + \nabla v \nabla w \Big) +t^2 \nabla w \nabla z \Big) dx.
			\end{equation*}
			So, 
			\begin{equation*}
				\lim\limits_{t \rightarrow 0} \dfrac{1}{|t|} \int_{\Omega} \Big(\nabla (u+tw) \nabla (v +t z)- \nabla u \nabla v\Big) dx = \int_{\Omega} \Big(\nabla u \nabla z + \nabla v \nabla w\Big) dx.
			\end{equation*}
			It then follows from the Dominated Convergence Theorem that $J$ has Gateaux derivative $J'$ at $(u,v)$ and
			\begin{equation*}
				\langle J'(u,v), (w,z) \rangle = \int_{\Omega} \Big(\nabla u \nabla z +\nabla v \nabla w -\lambda u w  -\delta v z -f(x,u) w - g(x,v)z \Big)dx.
			\end{equation*}
			\textbf{Continuity of the derivative} : Let $(u_n, v_n) \subset X$ such that $(u_n,v_n) \rightarrow (u,v)$ in $X$. By the continuous Sobolev embedding theorem, $(u_n,v_n) \rightarrow (u,v)$ in $L^p(\Omega)\times L^p(\Omega)$ with $1< p \le 2^{*}$. We have:
			\begin{align*}
				\Big|\langle J'(u_n,v_n)- J'(u,v), (w,z) \rangle \Big| &= \Big|\int_{\Omega} \Bigg (\nabla (u_n-u)\nabla z
				+ \nabla (v_n-v)\nabla w -\lambda (u_n-u)w -\delta (v_n -v) \\
				- & \Big(f(x,u_n)-f(x,u)\Big)w -\Big(g(x,v_n)-g(x,v)\Big)z \Bigg)\; dx \Big|\\
				\le & \Big|\langle (u_n-u, v_n-v), (z,w)\rangle \Big| + \lambda \int_{\Omega} |u_n-u||w| dx + \delta \int_{\Omega} |v_n-v||z| dx \\
				+& \int_{\Omega}|f(x,u_n)-f(x,u)||w|dx + \int_{\Omega}|g(x,v_n)-g(x,v)||z|dx.
			\end{align*}
			By the Theorem A.2 in \cite{Wi}, \[f(x,u_n) \rightarrow f(x,u),\; g(x,v_n)\rightarrow g(x,v)\quad \text{in}\quad L^r (\Omega), \;r:= \dfrac{p}{p-1}.\] 
			Hence, again by the Hölder inequality we have
			\begin{align}
            \label{eq 15}
				\Big|\langle J'(u_n,v_n)- J'(u,v), (w,z) \rangle \Big| &\le  (\|u_n-u\| + \|v_n-v\|) \|(z,w)\| +\lambda |u_n-u|_2 |w|_2 \nonumber\\
				& + \delta |v_n-v|_2|z|_2 + |f(x,u_n)-f(x,u)|_r |w|_p \nonumber\\
				& + |g(x,v_n)-g(x,v)|_r |z|_p \; \rightarrow 0, \quad n \rightarrow \infty.
			\end{align}
			\end{proof}
		We refer to \cite{KS} or \cite{Wi}  for the proof of the following lemma.
		\begin{lemma}
			\label{lemme 4.2}
		If the assumption $(H_3)$ is satisfied, then there exists $c_1 >0$ such that
			\begin{equation*}
				F(x,t) \ge c_1 (|t|^{\mu}-1); \qquad G(x,t) \ge c_1 (|t|^{\mu}-1).
			\end{equation*}
		\end{lemma}

		\begin{lemma}
			\label{lemme 4.3}
		The functional $J$ given in $(\star)$ satisfies the Palais-Smale condition for all $c\in \mathbb{R}$.
		\end{lemma}
		\begin{proof}
			Let $(u_n,v_n) \subset X$ such that $d:=\sup\;J(u_n,v_n)< \infty$ and $J'(u_n,v_n)\rightarrow 0$ as $n \rightarrow \infty$. \\
			We shall show that the sequence $(u_n,v_n)$ is bounded.\\
			Using Lemma $\ref{lemme 4.1}$, we have
			\begin{equation*}
				J(u_n,v_n)- \dfrac{1}{2} \langle J'(u_n,v_n), (u_n,v_n)\rangle =\int_{\Omega}\Bigg( \Big(\dfrac{1}{2} f(x,u_n)u_n -F(x,u_n)\Big) + \Big(\dfrac{1}{2} g(x,v_n)v_n -G(x,v_n)\Big)\Bigg) dx.
			\end{equation*}
			By Lemma $\ref{lemme 4.2}$, the assumption $(H_3)$ and for $n$ large enough, we obtain
			\begin{eqnarray*}
				d + \|(u_n,v_n)\| &\ge& \Big( \dfrac{\mu}{2} -1\Big)\int_{\Omega}\Big( F(x,u_n) + G(x, v_n)\Big) dx \nonumber \\ 
				&\ge& \Big( \dfrac{\mu}{2} -1\Big) \Bigg( \int_{\Omega} (c_1(|u_n|^\mu + |v_n|^\mu) dx - c_2 |\Omega| \Bigg).
			\end{eqnarray*}
			This implies that
			\begin{equation}
				\label{eq 16}
				|u_n|_{\mu}^{\mu} + |v_n|_{\mu}^{\mu} \le C_1 \|(u_n,v_n)\| + C_2,
			\end{equation}
			for some positives constants $C_1$ and $C_2$. \\
			On the other hand, for every $\epsilon>0$ and for $n$ large enough, we have
			\begin{align*}
				\|Q(u_n,v_n)\|^2 -\epsilon \|Q(u_n,v_n)\|
               &\le \Big|\|Q(u_n,v_n)\|^2 - \dfrac{1}{2} \langle J'(u_n,v_n), (u_n+v_n,u_n +v_n)\rangle \Big| \\
				&= \Big|\int_{\Omega} \Big(\lambda u_n + \delta v_n + f(x,u_n) + g(x,v_n)\Big) \Big(\dfrac{u_n +v_n}{2}\Big)dx\Big| \\
				&\le \lambda |u_n|_2|Q(u_n,v_n)|_{2\times 2} + \delta |v_n|_2|Q(u_n,v_n)|_{2\times 2}\\
				 &+ |f(x,u_n)|_r|Q(u_n,v_n)|_{p\times p} + |g(x,v_n)|_r|Q(u_n,v_n)|_{p\times p}.
			\end{align*}
			Since $\mu >2$, then we have that $|u_n|_{2}^{2} \le |u_n|_{\mu}^{\mu}$, that is, $|u_n|_2 \le |u_n|_{\mu}^{\mu /2}$. From the Sobolev embedding theorem, there exists $k_1>0$ such that
			\begin{equation*}
				|Q(u_n,v_n)|_{p\times p} \le k_1 \|Q(u_n,v_n \|.
			\end{equation*}
			Thus, we have 
			\begin{eqnarray*}
				\|Q(u_n,v_n)\|^2 -\epsilon \|Q(u_n,v_n)\| &\le& K \|Q(u_n,v_n)\| \Big(|u_n|_{\mu}^{\mu/2} + |v_n|_{\mu}^{\mu/2} + |f(x,u_n)|_r + |g(x,v_n)|_r\Big),
			\end{eqnarray*}
			for some positive constant $K$. \\
			Similarly, we have
			\begin{eqnarray*}
				\|P(u_n,v_n)\|^2 -\epsilon \|P(u_n,v_n)\| &\le& \Big|-\|P(u_n,v_n)\|^2 - \dfrac{1}{2} \langle J'(u_n,v_n), (u_n-v_n,v_n -u_n)\rangle \Big| \\
				& \le& K \|P(u_n,v_n)\| \Big(|u_n|_{\mu}^{\mu/2} + |v_n|_{\mu}^{\mu/2} + |f(x,u_n)|_r + |g(x,v_n)|_r\Big).
			\end{eqnarray*}
			Therefore, 
			\begin{eqnarray}
				\|Q(u_n,v_n)\| -\epsilon &\le& K \Big(|u_n|_{\mu}^{\mu/2} + |v_n|_{\mu}^{\mu/2} + k_2 \Big), \\
				\|P(u_n,v_n)\| -\epsilon &\le& K \Big(|u_n|_{\mu}^{\mu/2} + |v_n|_{\mu}^{\mu/2} + k_2 \Big),
			\end{eqnarray}
			with $0<k_2:= |f(x,u_n)|_r + |g(x,v_n)|_r<\infty$.\\
		The preceding inequalities imply that 
			\begin{equation*}
				\|(u_n,v_n)\| -2 \epsilon \le k_3 +k_4 \Big(|u_n|_{\mu}^{\mu/2} + |v_n|_{\mu}^{\mu/2} \Big),
			\end{equation*}
			for some positives constants $k_3$ and $k_4$. \\
			Moreover, inequality $(\ref{eq 16})$ implies that 
			\begin{eqnarray*}
				|u_n|_\mu &\le&  k_5 + \|(u_n,v_n)\|^{\frac{1}{\mu}}, \\
				|v_n|_\mu &\le&  k_6 + \|(u_n,v_n)\|^{\frac{1}{\mu}},
			\end{eqnarray*}
			for some positives constants $k_5$ and $k_6$. \\
			It follows that
			\begin{equation*}
				\|(u_n,v_n)\| -2 \epsilon \le D_1 \|(u_n,v_n)\|^{\frac{1}{2}} + D_2,
			\end{equation*} 
			for some positives constants $D_1, D_2$.\\ 
			Hence $(u_n,v_n)$ is bounded in $X$. \\
		 Since $(u_n,v_n)\subset X$ is bounded, $(u_n,v_n)$ possesses a subsequence (we call it again $(u_n,v_n)$) which converges weakly. That is, there exists $(u,v) \in X$ such that $(u_n,v_n) \rightharpoonup (u,v)$. We have
			
			\begin{align*}
				\langle J'(u_n,v_n)- J'(u,v), (u_n-u,0 )\rangle &= \|u_n-u\|^2 + \langle (u_n-u,0), (v_n-v -u_n +u,0)\rangle \\
				-&\lambda \int_{\Omega} (u_n-u)^2dx - \int_{\Omega} \Big(f(x,u_n) -f(x,u)\Big)(u_n-u) dx.
			\end{align*}
			By the Rellich-Kondrachov embedding theorem, $u_n \rightarrow u$ in $L^p(\Omega)$ and $v_n \rightarrow v$ in $L^p(\Omega)$. The Hölder inequality implies that
			\begin{eqnarray*}
				\Big|\int_{\Omega} (u_n-u)^2dx\Big| &\le& |u_n-u|_2|u_n-u|_2\quad \rightarrow 0, \quad n\rightarrow \infty,\\
				\Big|\int_{\Omega} \Big(f(x,u_n) -f(x,u)\Big)(u_n-u) dx\Big|	&\le& |f(x,u_n)-f(x,u)|_r|u_n-u|_p \quad \rightarrow 0, \quad n \rightarrow \infty.
			\end{eqnarray*}
			Since $(u_n,v_n)\rightharpoonup (u,v)$ and $J'(u_n, v_n) \rightarrow 0$, then
			\begin{eqnarray*}
				\langle J'(u_n,v_n)- J'(u,v), (u_n-u,0 )\rangle \quad \rightarrow 0, \; n\rightarrow \infty, \\
				\langle (u_n-u,0), (v_n-v -u_n +u,0)\rangle \; \rightarrow 0, \; n \rightarrow \infty.
			\end{eqnarray*}
			Thus, \[\|u_n-u\|^2 \rightarrow  0, \; n\rightarrow \infty.\]
			By the same way we show that $\|v_n-v\|^2 \rightarrow  0, \; n\rightarrow \infty$. We deduce that $J$ satisfies the Palais-Smale condition. The proof of  Lemma $\ref{lemme 4.3}$ is complete.
		\end{proof}

	The next two results show that the functional $J$ given in $(\star)$ satisfies the geometric assumptions $(\mathcal{G})$ of Theorem $\ref{theorem 3.1}$.
		
		\begin{lemma}
			\label{lemme 4.4}
			There exists $r>0$ such that $\underset{(u,u)\in N}{\inf}\;J(u,u) >0$, where $N$ is given by $(\ref{eq 9})$.
		\end{lemma}
		\begin{proof}
			Assumptions $(H_1)$ and $(H_2)$ imply that (see \cite{Wi}) for every $\epsilon>0$, there exists $c_\epsilon>0$ such that 
			\begin{equation}
				\label{eq 19}
		 |F(x,t)|, |G(x,t)|\le \dfrac{\epsilon}{2} |t|^2 + c_\epsilon |t|^p,\;\; \text{for all}\;\;t \in \mathbb{R}.
			\end{equation}
		Thus, on $N$, we have
			\begin{eqnarray*}
				J(u,u) &=& \dfrac{\|Q(u,u)\|^2}{2} - \varphi(u,u)\\
				&=& \|u\|^2 -\varphi(u,u),
			\end{eqnarray*}
			where
			\begin{equation*} 
				\varphi(u,u) := \int_{\Omega} \Bigg(\dfrac{\lambda u^2}{2} +\dfrac{\delta u^2}{2}+ F(x,u) + G(x,u)\Bigg) \; dx.
			\end{equation*}
		The inequality $(\ref{eq 19})$ implies that, for every $\epsilon >0$, there is $c_\epsilon>0$ such that
			\begin{equation*}
				J(u,u) \ge \|u\|^2 - (\kappa +\epsilon) |u|_{2}^{2} -2c_\epsilon |u|_{p}^p.
			\end{equation*}
			where $\kappa := \dfrac{\lambda}{2} +\dfrac{\delta}{2}$.\\
			From the Sobolev embedding theorem, namely $ H_{0}^1 (\Omega) \hookrightarrow{} L^p(\Omega), 1<p\le2^* \; ( $that is$,\; $there exists$\: c_0 >0\; $such that$\; |u|_p \le c_0 \|u\|)$, we obtain
			\begin{equation*}
				J(u,u) \ge \Big(1- (\kappa + \epsilon) \bar{c_0}\Big) \|u\|^2 - 2c_\epsilon \bar{c_0} \|u\|^p,
			\end{equation*} 
			for some positive constant $\bar{c_0}$. \\
			 Choosing $(\kappa + \epsilon)\bar{c_0} =\dfrac{1}{2}$, we have 
			\begin{equation*}
J(u,u) \ge \dfrac{1}{2} \|u\|^2 - 2c_\epsilon \bar{c_0} \|u\|^p.
			\end{equation*}
			We can choose $r>0$ sufficiently small, such that 
			\begin{equation*}
		J(u,u) >0, \;\text{whenever} \; \|u\|=r.
			\end{equation*} 
		\end{proof}
		\begin{lemma}
			\label{lemma 4.5}
			Let $(z,z)\in Z$ such that $\|(z,z)\| =r$. There exists $\rho>r>0$ such that $\underset{\partial M}{\sup}\;J \le0$, where $\partial M$ is given by $(\ref{eq 8})$.
		\end{lemma}
		\begin{proof}
			By assumption $(H_3)$, we have on $Y$
			\begin{equation*}
				J(-v,v) = -\dfrac{\|P(-v,v)\|^2}{2} - \varphi(-v,v) = -\|v\|^2 -\varphi(-v,v) \le 0.
			\end{equation*}
			By Lemma $\ref{lemme 4.2}$, we have
			\begin{align*}
				J((-v,v)+\lambda_0(z,z))&\le -\|v\|^2 -\dfrac{\lambda}{2}|-v+\lambda_0 z|_{2}^2  -\dfrac{\delta}{2}|v+\lambda_0 z|_{2}^2 + \dfrac{\lambda_{0}^2}{2} \\
				& + c_1 \Big(2|\Omega|- |-v+\lambda_0 z|_{\mu}^{\mu} - |v+\lambda_0 z|_{\mu}^{\mu}\Big) \\
				&\le -\|v\|^2 +\dfrac{\lambda_{0}^2}{2}+ 2c_1 |\Omega|- c_1|-v+\lambda_0 z|_{\mu}^{\mu} - c_1|v+\lambda_0 z|_{\mu}^{\mu}.
			\end{align*}
			Let $W$ be the closure of $Y\oplus \mathbb{R}(z,z)$ in $L^{\mu}(\Omega)\times L^{\mu}(\Omega) $. Since the Sobolev space $H_{0}^1 (\Omega) \times H_{0}^1 (\Omega) $ embeds continuously into $L^{\mu}(\Omega)\times L^{\mu}(\Omega)$, then there exists a continuous projection from $W$ onto $\mathbb{R}(z,z)$. Moreover, since all norms are equivalent in a finite-dimensional vector space , there exist some constants $c_3>0$ and $c_4>0$ such that  $c_3|\lambda_{0} z|_\mu \le |-v+\lambda_{0} z|_\mu$, $c_4|\lambda_{0} z|_\mu \le |v+\lambda_{0} z|_\mu$. \\
			Therefore we have
			\begin{eqnarray*}
				J\Big((-v,v)+\lambda_0(z,z)\Big) &\le& -\|v\|^2 +\dfrac{\lambda_{0}^2}{2}+ 2c_1 |\Omega| - c_1(c_3^{\mu} + c_4^{\mu})|\lambda_{0}z|_{\mu}^{\mu},\\
				&\le& -\|v\|^2 +\dfrac{\lambda_{0}^2}{2}+ 2c_1 |\Omega| - c_5\lambda_{0}^{\mu},
			\end{eqnarray*}
		for some constant $c_5>0$.\\
			Since $\mu >2$, it follows that for $w \in \partial M$
			\begin{equation*}
				J(w) \rightarrow -\infty \;\text{whenever}\;\|w\| \rightarrow \infty,
			\end{equation*}
			and so, taking $\rho>r$ large, we get $\underset{\partial M}{\sup}\;J \le0$.
		\end{proof}

        The following result shows that the functional $J$ given in $(\star)$ satisfies the assumption $(A)$ of Theorem $\ref{theorem 3.1}$.
        \begin{lemma}
        \label{lem 4.6}
         The functional $J$ is $\tau-$upper semicontinuous and $\nabla J$ is weakly sequentially continuous.
        \end{lemma}
        \begin{proof}
            1. Let us show that \textbf{ $J$ is $\tau-$upper semicontinuous}.
	For every $c\in \mathbb{R}$, let show that the set \[ \Big \{(u,v)\in X\,|\, J(u,v)\ge c \Big\}\quad \text{is}\quad\tau-closed.\]
	Let $(u_n,v_n) \subset X$ such that $(u_n,v_n)\overset{\tau}{\rightarrow} (u,v)$ in $X$ and $c\le J(u_n,v_n)$. Then, $(u_n, v_n)$ is bounded. Thus, by $(\ref{eq 4})$, we derive that $(u_n,v_n) \rightharpoonup (u,v)$ in $X$. By the Rellich-Kondrachov embedding theorem, up to a subsequence we have 
	\begin{eqnarray*}
		u_n &\to& u \quad \text{in}\quad L^p(\Omega),\\
        v_n &\to& v \quad \text{in}\quad L^p(\Omega), \\
        u_n(x) &\to& u(x), \;v_n(x) \to v(x), \quad \text{almost everywhere} ,\quad \text{in}\quad \Omega,\\
        F(x,u_n(x)) &\to& F(x,u(x)), \quad \text{almost everywhere} ,\quad \text{in}\quad \Omega,\\
        G(x,u_n(x)) &\to& G(x,u(x)), \quad \text{almost everywhere} ,\quad \text{in}\quad \Omega.
	\end{eqnarray*}
Since $F(x,u_n)\ge 0$ and $G(x,v_{n})\ge 0$, then by the Fatou Lemma, we have 
	\begin{equation*}
		\int_\Omega F(x,u) dx= \int_ \Omega\underset{n\rightarrow \infty}{\underline{\lim}}F(x,u_{n})dx  \le \underset{n\rightarrow \infty}{\underline{\lim}} \int_\Omega F(x,u_{n})dx.
	\end{equation*}
    In the same way, we have 
    \begin{equation*}
		\int_\Omega G(x,v) dx= \int_ \Omega\underset{n\rightarrow \infty}{\underline{\lim}}G(x,v_{n})dx  \le \underset{n\rightarrow \infty}{\underline{\lim}} \int_\Omega G(x,v_{n})dx.
	\end{equation*}
	Since $\|\cdot\|$ is weak lower semi-continuous, we have 
	\begin{equation*}
		\|P(u,v)\|^2 \le \underset{k\rightarrow \infty}{\underline{\lim}} \|P(u_n,v_n)\|^2.
        \end{equation*}
       	Moreover, since $\|Q(u_n,v_n)\|^2 \to \|Q(u,v)\|^2$, we have 
		\begin{equation*}
		 \underset{k\rightarrow \infty}{\overline{\lim}} \Big( -\|Q(u_n,v_n)\|^2\Big)= 
			-\|Q(u,v)\|^2 = 	\underset{k\rightarrow \infty}{\underline{\lim}} (-\|Q(u_n,v_n)\|^2).
		\end{equation*}
	Hence, 
	\begin{eqnarray*}
		-J(u,v) &=& \dfrac{1}{2}\Big( \|P(u,v)\|^2-\|Q(u,v)\|^2 \Big) +\varphi (u,v)\\ &\le& \underset{k\rightarrow \infty}{\underline{\lim}} \Bigg(\dfrac{1}{2}\Big( \|P(u_n,v_n)\|^2-\|Q(u_n,v_n)\|^2 \Big) +\varphi (u_n,v_n)  \Bigg)\\
		&=&  \underset{k\rightarrow \infty}{\underline{\lim}} (-J(u_n,v_n) )\\
		&=& - \underset{k\rightarrow \infty}{\overline{\lim}} J(u_n,v_n)\\
		&\le& -c.
	\end{eqnarray*}
    2. Now, let us show that $\nabla J$ is weakly sequentially continuous. Let $(u_n,v_n) \subset X$ such that $(u_n,v_n) \rightharpoonup (u,v)$ in $X$. We have to show that \[\nabla J(u_n,v_n) \rightharpoonup \nabla J(u,v).\]
    For every $(w,z)\in X$, we have established that (see ($\ref{eq 15}$))
     \begin{equation*}
         \Big|\langle J'(u_n,v_n)- J'(u,v), (w,z) \rangle \Big| \quad \rightarrow 0, \quad n \rightarrow \infty.
     \end{equation*}
     Therefore, for every $(w,z) \in X$,
     \begin{equation*}
        \langle \nabla J(u_n,v_n), (w,z) \rangle \to \langle \nabla J(u,v), (w,z) \rangle.
     \end{equation*}
     Hence, $\nabla J(u_n,v_n) \rightharpoonup \nabla J(u,v)$.
        \end{proof}
	
	\subsection*{Proof of Theorem $\ref{thm 4.1}$}
By Lemma $\ref{lemme 4.1}$, Lemma $\ref{lemme 4.3}$ and Lemma \ref{lem 4.6}, the functional $J \in \mathcal{C}^1(X, \mathbb{R})$, satisfies $(PS)_c$ condition for all $c \in \mathbb{R}$, $J$ is $\tau-$upper semicontinuous and $\nabla J$ is weakly sequentially continuous.  \\
By Lemma $\ref{lemme 4.4}$ and Lemma $\ref{lemma 4.5}$, $J(u,v)$ satisfies the geometric assumptions $(\mathcal{G})$ of Theorem $\ref{theorem 3.1}$. So, invoking Theorem $\ref{theorem 3.1}$, there exists $(u,v)\in X$ such that $J'(u,v) =0$ and $J(u,v) =c$. The existence of a critical point of $J$, that is, a solution to problem $(\mathcal{P})$ is now established, finishing the proof. $\square$
	\section*{Conclusion}
In this work, we presented a natural generalization of the linking theorem to strongly indefinite functionals. Taking advantage of the general principle of minimax found in \cite{CS}, we gave a straightforward proof of our main result. The same ideas will be used in future work to generalize other critical point theorems for strongly indefinite functionals.

\end{document}